\documentclass[12pt]{article}

\usepackage{enumerate}
\usepackage{amsthm,amsmath,amssymb,url}
\usepackage{graphicx}
\usepackage{tikz}
\usetikzlibrary{graphs}
\usepackage{tkz-graph} 
\usepackage[colorlinks=true,citecolor=black,linkcolor=black,urlcolor=blue]{hyperref}

\theoremstyle{plain}
\newtheorem{theorem}{Theorem}
\newtheorem{lemma}[theorem]{Lemma}
\newtheorem{corollary}[theorem]{Corollary}
\newtheorem{proposition}[theorem]{Proposition}

\theoremstyle{definition}
\newtheorem{definition}[theorem]{Definition}
\newtheorem{example}[theorem]{Example}

\theoremstyle{remark}
\newtheorem{remark}[theorem]{Remark}

\DeclareMathOperator*{\tr}{tr}

\def\Z{{\mathbb{Z}}}
\def\R{{\mathbb{R}}}

\def\Z{{\mathbb{Z}}}
\def\og{{\Omega}}

\def\cP{{\cal P}}

\title{\bf Orbigraphs: A Graph Theoretic Analog to Riemannian Orbifolds}

\begin{document}

\maketitle

\begin{small}
\begin{center}
\noindent Kathleen Daly, Booz Allen Hamilton, daly\_kathleen@bah.com\\
\smallskip

Colin Gavin, 908 Devices, cgavin@908devices.com\\
\smallskip

Gabriel Montes de Oca,  Department of Mathematics \\ University of Oregon, gabem@uoregon.edu\\
\smallskip

Diana Ochoa, Department of Mathematical Sciences \\ Lewis and Clark College, dochoa@lclark.edu\\
\smallskip

Elizabeth Stanhope, Department of Mathematical Sciences \\ Lewis and Clark College, stanhope@lclark.edu\\
\smallskip

Sam Stewart,  School of Mathematics \\ University of Minnesota,  sams@umn.edu\\

\medskip

\end{center}
\end{small}

\begin{abstract} 
A Riemannian orbifold is a mildly singular generalization of a Riemannian manifold that is locally modeled on $R^n$ modulo the action of a finite group. Orbifolds have proven interesting in a variety of settings.  Spectral geometers have examined the link between the Laplace spectrum of an orbifold and the singularities of the orbifold.  One open question in this field is whether or not a singular orbifold and a manifold can be Laplace isospectral. Motivated by the connection between spectral geometry and spectral graph theory, we define a graph theoretic analogue of an orbifold called an orbigraph.  We obtain results about the relationship between an orbigraph and the spectrum of its adjacency matrix. We prove that the number of singular vertices present in an orbigraph is bounded above and below by spectrally determined quantities, and show that an orbigraph with a singular point and a regular graph cannot be cospectral.  We also provide a lower bound on the Cheeger constant of an orbigraph.  
\end{abstract}

\subsection*{Acknowledgements}
This work was supported in part by the Rogers Summer Research Program and Fairchild Program at Lewis \& Clark College.  The authors would like to thank Omar Lopez and Yung-Pin Chen for foundational work with orbigraphs and discussions of Markov processes, respectively.  We also thank the reviewer for their helpful suggestions.

\section{Introduction}

 A Riemannian orbifold is a mildly singular generalization of a Riemannian manifold.  A point in an $n$-dimensional manifold is contained in a neighborhood that is homeomorphic to $\R^n$.   A point in an $n$-dimensional orbifold is contained in a neighborhood that is homeomorphic to a quotient of $\R^n$ under the action of a finite group.  Two useful examples of orbifolds to consider are the $\Z_n$-football (Figure \ref{fig1}) and the $\Z_n$-teardrop (Figure \ref{fig2}):

\begin{example}\label{exa:football} Let $\Z_n$ act on a two-dimensional sphere by rotations generated by a $2\pi/n$ radian rotation about an axis passing through the center of the sphere.  The quotient of the sphere under this action is the $\Z_n$-football.  Points lying on the intersection of the sphere with the axis of rotation are fixed by all rotations.  The images in the $\Z_n$-football of these points are the conical points at the north and south poles of the football.  If the local lift of a point in an orbifold has non-trival isotropy, the point is called a \emph{singular point} in the orbifold.  The singular set of the $\Z_n$-football consists of the cone points at its north and south poles.
\end{example}

\begin{example}\label{exa:teardrop}  The $\Z_n$-teardrop is topologically a 2-sphere except for a single point whose neighborhood is locally modeled on the cone $\R^2/\Z_n$, where $\Z_n$ acts by rotations around a fixed point.  Thus the $\Z_n$-teardrop's singular set consists of the isolated cone point. Thurston \cite{thur} showed that unlike the $\Z_n$-football, the $\Z_n$-teardrop cannot be obtained as the quotient of a manifold under a smooth, discrete group action. 

\end{example}

\begin{figure}\label{ofds}
\centering
\begin{minipage}{.35\textwidth}
  \centering
 \includegraphics[width=1.5in]{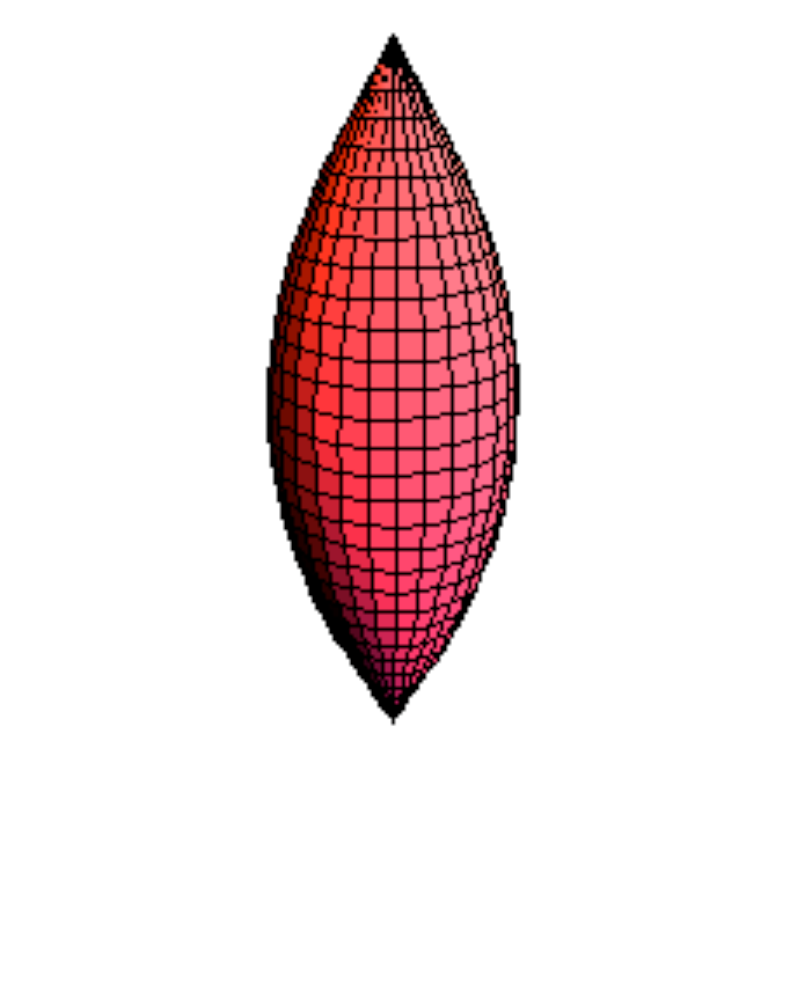}   
  \caption{Football obtained by 180 degree rotation of sphere}\label{fig1}
   \label{football}
\end{minipage}
\hspace{2cm}
\begin{minipage}{.35\textwidth}
  \centering
  \includegraphics[width=1.25in]{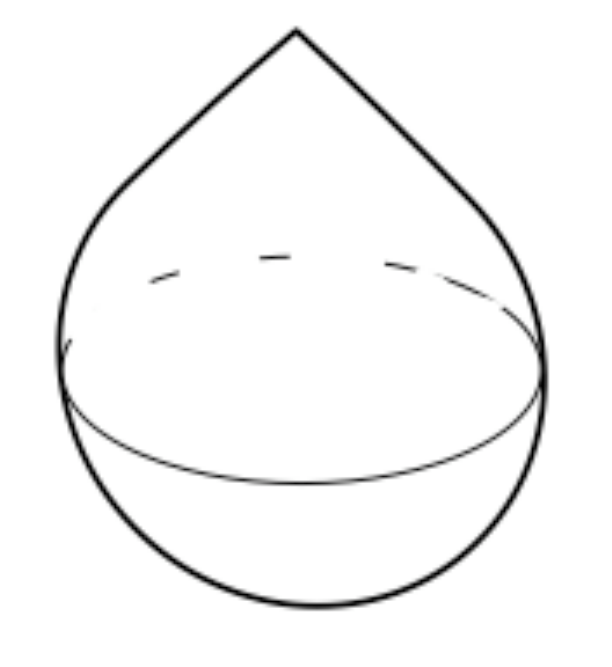}
   \caption{Teardrop orbifold}\label{fig2}
  \end{minipage}
  \hfill
\end{figure}

Introduced by Satake \cite{MR18:144a} in 1956 under the name $V$-manifold, and later renamed and studied as orbifolds by Thurston \cite{thur}, orbifolds have proven interesting in a variety of settings ({see \cite{MR2359514, gordon12, MR93m:58114}, for example).  Of particular interest are results relating the eigenvalue spectrum of the Laplace operator on a Riemannian orbifold (an orbifold endowed with a suitably invariant Riemannian metric) to the singular set of the orbifold.  For example, in the presence of a curvature hypothesis, the fifth author showed \cite{MR2155380} that the Laplace spectrum constrains the structure of the singular set.  One fundamental orbifold spectral geometry question that remains open is whether or not the Laplace spectrum actually detects the presence of singular points.
    
Brooks \cite{MR1191744, MR1697378} proposes viewing $k$-regular graphs as combinatorial analogs of smooth manifolds.  The infinite $k$-regular tree $T_k$ is viewed as the graph theoretic version of the universal cover of a finite $k$-regular graph.  A finite $k$-regular graph $\Gamma$ is studied as the quotient of $T_k$ by the fundamental group of $\Gamma$ in analogy to the study of quotients of the universal cover of a manifold under the action of a discrete co-compact group of isometries acting freely.  In this setting Brooks obtains several results including a characterization of Ramanujan graphs, a partial converse to Sunada's Theorem, and links between the spectrum of a $k$-regular graph and the graph's diameter and girth.

Following Brooks's analogy, observe that the action of a discrete, co-compact group of isometries which is not free yields a quotient space that is an orbifold rather than a manifold.  Given the successful examination of orbifolds from the perspective of spectral geometry, we seek to extend Brooks' analogy one step further by first proposing a graph theoretic analog of an orbifold and, second, applying the lense of spectral graph theory to orbifold graphs. References in the literature to an orbifold-like class of graphs are limited.  Brooks \cite{MR1697378} himself describes an ``orbifold graph" as a quotient of a $k$-regular graph under a non-free group action. He offers orbifold graphs as a motivating idea, but chooses to ``avoid entering into the technicalities of `orbifold graphs.'"  Lafont \cite{MR2855898}  describes an analogy between orbifolds and objects from Bass-Serre Theory \cite{MR1239551} called \emph{graphs-of-groups}.  Although the present work has its roots in the ideas of Brooks, the graphs that we examine here can be viewed as a generalization of the edge-index graph of a graph of groups.

We define an \emph{orbigraph} to be a member of the following class of weighted, directed graphs.

 \begin{definition}\label{def:OrbigraphDefn}
    An \textbf{orbigraph of degree $k$ ($k$-orbigraph)} is a finite, weighted, directed graph $\Omega$ where the adjacency matrix $A$ of $\Omega$ satisfies the following:
    
    \begin{enumerate}[(i)]
      \item $A_{ij} \in \mathbb{Z}_{\ge 0}$
      \item $\sum_j A_{ij} = k$
      \item $A_{ij} > 0$ if and only if $A_{ji} > 0$.
    \end{enumerate}
  \end{definition}
  
\begin{figure}
\centering
\begin{minipage}{.35\textwidth}
 \centering
\begin{tikzpicture}
\GraphInit[vstyle=Normal]
\SetGraphUnit{3}
\Vertex[Math]{v_1}
\EA[Math](v_1){v_2}
\Edge[style={->,bend left}](v_1)(v_2)
\Edge[style={->,bend left},label=3,labelstyle=above](v_2)(v_1)
\Loop[dist=1.5cm,dir=WE,style={->},labelstyle={fill=white},label=2](v_1)
\end{tikzpicture} 
\caption{a small 3-orbigraph}
\label{smallexa}
\end{minipage}
\hfill
\begin{minipage}{.6\textwidth}
  \centering
 \begin{tikzpicture}
[scale=.9]
\GraphInit[vstyle=Normal]
\SetGraphUnit{3}
\begin{scope}[rotate=51.5]
\Vertices[Math]{circle}{v_1,v_2,v_3,v_4,v_5,v_6,v_7}
\Edge[style={->,bend right},label=2,labelstyle={fill=white}](v_1)(v_2)
\Edge[style={->,bend right}](v_1)(v_7)
\Edge[style={->,bend right}](v_2)(v_1)
\Edge[style={->,bend right}](v_2)(v_3)
\Edge[style={->,bend left}](v_2)(v_5)
\Edge[style={->,bend right}](v_3)(v_4)
\Edge[style={->,bend right}](v_3)(v_7)
\Edge[style={->,bend right}](v_3)(v_2)
\Edge[style={->,bend right}](v_4)(v_5)
\Edge[style={->,bend right},label=2,labelstyle={fill=white}](v_4)(v_3)
\Edge[style={->,bend right}](v_5)(v_4)
\Edge[style={->,bend right}](v_5)(v_6)
\Edge[style={->}](v_5)(v_2)
\Edge[style={->,bend right}](v_6)(v_7)
\Edge[style={->,bend right},label=2,labelstyle={fill=white}](v_6)(v_5)
\Edge[style={->,bend right}](v_7)(v_6)
\Edge[style={->}](v_7)(v_3)
\Edge[style={->,bend right}](v_7)(v_1)
\end{scope}
\end{tikzpicture}
\caption{a 3-orbigraph with 7 vertices}
\label{bigexa}
\end{minipage}
\end{figure}

Figures \ref{smallexa} and \ref{bigexa} show two examples of orbigraphs.   

\begin{remark}\label{connected}
All orbigraphs discussed below will be assumed to be connected unless noted otherwise. Condition (iii) in Definition~\ref{def:OrbigraphDefn} implies that a connected orbigraph must be strongly connected. 
Nonzero diagonal entries in the adjacency matrix of an orbigraph correspond to weighted loops in the orbigraph.
\end{remark}

In Section \ref{topology} below we demonstrate the analogy between orbigraphs and orbifolds through the following three points:
\begin{itemize}
\item[a.] The local structure of a vertex in a $k$-orbigraph is that of the quotient of a $k$-regular graph just as the local structure of a $k$-dimensional orbifold is the quotient of a $k$-dimensional manifold.
\item[b.] Some vertices in an orbigraph have the same local structure as a vertex in a regular graph and some do not.  This leads us to the definition of regular and singular vertices in an orbigraph -- an essential piece of the analogy between orbifolds and orbigraphs.
\item[c.] We show that some orbigraphs can be obtained as the quotient of a finite regular graph under an equitable partition and some cannot.  This mirrors the fundamental fact from the geometric setting that orbifolds are divided into two classes: those that are covered by a manifold (like the football) and those that are not (like the teardrop).  Indeed, the presence of singular objects that are not merely quotients of regular objects saves the study of orbifolds and orbigraphs from being simply a reduced version of a known field of study.
\end{itemize}

Section~\ref{markov} connects orbigraphs to the theory of Markov chains.  In Section~\ref{bcc} Markov chain methods are used to obtain a graph theoretic characterization of when an orbigraph can be obtained as the quotient of a finite regular graph, and when it cannot. This characterization makes it easy to generate examples of orbigraphs with these properties, facilitating our later examination of how spectral results for orbifolds carry over to the orbigraph setting.   Also using Markov chain methods we provide a lower bound on the Cheeger constant of a $k$-orbigraph in terms of $k$ and the size of its vertex set.  This adds a third family to the list in Chung \cite{MR2135772} of families of directed graphs that satisfy similar bounds.  It would be interesting to know if the bound presented here is sharp, or if an improved bound could be used to obtain a strong upper bound on the convergence of random walks on orbigraphs. Our examination of the Cheeger constant on orbigraphs is the topic of Section~\ref{cheeger}.

In Section \ref{spectral} we follow the philosophy of Brooks and ask questions from the spectral geometry of orbifolds in the orbigraph setting.  The orbigraph spectrum discussed here is the list of eigenvalues of the adjacency matrix of an orbigraph.
 Because the analogy between orbifolds and orbigraphs extablished in Section~\ref{topology} is strong, the questions carry over naturally and we obtain several interesting results:

\begin{itemize}
\item[a.] We show that the spectrum does not detect whether or not an orbigraph can be obtained as the quotient of a finite $k$-regular graph.  The analogous question for orbifolds is still an open problem in spectral geometry.
\item[b.] The number of singular points in an orbigraph can be bounded both above and below by spectrally-determined quantities.  In the geometric setting one can seek spectral bounds on the number of components of the singular set.  In dimension two, the fifth author and Proctor \cite{MR2579379} obtained a result of this type under a curvature hypothesis.
\item[c.] The spectrum of an orbigraph detects the presence of singular points.  As mentioned above, this question is still open in the orbifold setting.
\end{itemize}

\section{Orbigraphs as discrete orbifolds}\label{topology}

\subsection{Local structure of a $k$-orbigraph}

The local structure of an orbigraph is that of a quotient of a $k$-regular graph.  There are multiple ways to define the quotienting process for graphs.  Here quotient graphs will be formed with respect to an equitable partition.  The definition given below uses the approach of Barrett, Francis and Webb \cite{MR3573808} to extend the definition of an equitable partition from the familiar setting of simple graphs to the more general setting of weighted directed graphs.   We also follow the thorough treatment of the simple graph case in Chapter 5 of Godsil \cite{MR1220704}.  

In what follows let $w(u,v)$ denote the weight of directed edge $(u,v)$.

\begin{definition}\label{def:EqPartitionDefn}  Let $\Gamma$ be an graph (possibly directed, weighted, or both) and $${\cal P} = \{V_1, V_2, \ldots, V_m \}$$ be a partition of its vertices. 
\begin{itemize}
\item[a.] We say ${\cal P}$ is an \emph{equitable partition} if for all pairs $i,j$ the number $\sum_{v \in V_j} w(u, v)$ is the same for each element $u$ in $V_i$. 
\item[b.] Given an equitable partition ${\cal P}$ on $\Gamma$, the weighted directed graph with adjacency matrix $A_{ij}=\sum_{v \in V_j} w(u, v)$, $u$ in $V_i$, is called the \emph{quotient graph} of $\Gamma$ with respect to ${\cal P}$ and will be denoted $\Gamma/\cal P$.
\end{itemize}
\end{definition}

\begin{remark}\label{rem:orbitpart} If a group $G$ acts on a simple graph $\Gamma$ by automorphisms, the vertex orbits of the action form an equitable partition of the vertex set of $\Gamma$.  This type of equitable partition is called an \emph{orbit partition}. In this case the quotient graph will be written $\Gamma/G$.

\end{remark}

To discuss the local structure of an orbigraph we introduce further terms from graph theory.  Note that a non-directed edge $\{v,w\}$ of weight $n$ in a graph will be viewed as being equivalent to a pair of weight $n$ directed edges $(v,w)$ and $(w,v)$, and vice-versa. 

\begin{definition}
\begin{itemize}
\item[a.] The \emph{$k$-star graph} is the complete bipartite graph $K_{1,k}$ and will be denoted $S_k$.  The vertex with degree $k$ in $S_k$ is the \emph{central vertex} of $S_k$.
\item[b.] The \emph{neighborhood} of a vertex $v$ in an undirected graph $\Gamma$ is the subgraph of $\Gamma$ including the vertex $v$, all vertices $w$ adjacent to $v$, and all edges $\{v,w\}$.  
\item[c.] The \emph{out-neighborhood} of a vertex $v$ in a directed graph $\Delta$ is the directed subgraph of $\Delta$ including vertex $v$, all vertices $w$ at which edges initiating at $v$ terminate, and all directed edges $(v,w)$ with initial vertex $v$. 
\end{itemize} 
\end{definition}

Because the neighborhood of each vertex in a simple $k$-regular graph is $S_k$ we view a simple $k$-regular graph as the graph theoretic analog of a $k$-dimensional manifold.  

Let $G$ be a group of graph automorphisms of $S_k$ and form the quotient graph $S_k/G$.  The central vertex $c$ of $S_k/G$ is the vertex in $S_k/G$ associated to the element of the orbit partition on $S_k$ containing the central vertex of $S_k$.  The out-neighborhood of $c$ in $S_k/G$ is a weighted star graph with between 1 and $k$ edges.  The sum of the weights over all edges in the out-neighborhood of $c$ is $k$.  

\begin{example}  There are only three different weighted, directed graphs that arise as quotients of $S_3$ by a group of graph automorphisms. Figure \ref{quotient3stars} illustrates the out-neighborhoods of the central vertex in each of these three quotients.
\end{example}
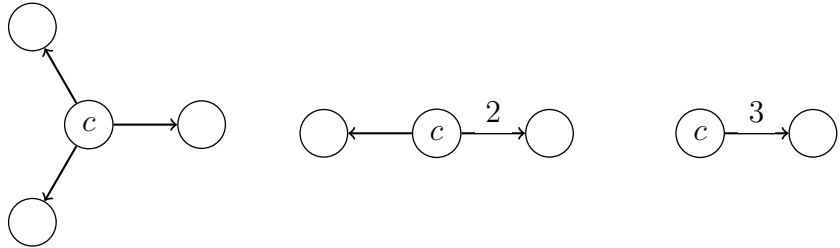
\begin{figure}
\centering
\begin{minipage}{.3\textwidth}
  \centering
 \begin{tikzpicture}
\GraphInit[vstyle=Normal]
\SetGraphUnit{1.5}
\begin{scope}[rotate=120]
\Vertices[Math,NoLabel]{circle}{x,y,z}
\Vertex[Math,x=0,y=0]{c}
\Edge[style={->}](c)(x)
\Edge[style={->}](c)(y)
\Edge[style={->}](c)(z)
\end{scope}
\end{tikzpicture}
\end{minipage}
\begin{minipage}{.3\textwidth}
  \centering
  \begin{tikzpicture}
\GraphInit[vstyle=Normal]
\SetGraphUnit{1.5}
\Vertex[Math,NoLabel]{x}
\EA[Math](x){c}
\EA[Math,NoLabel](c){y}
\Edge[style={->}](c)(x)
\Edge[label=2,labelstyle=above,style={->}](c)(y)
\end{tikzpicture}
\end{minipage}
\begin{minipage}{.3\textwidth}
  \centering
  \begin{tikzpicture}
\GraphInit[vstyle=Normal]
\SetGraphUnit{1.5}
\Vertex[Math]{c}
\EA[Math,NoLabel](c){x}
\Edge[label=3,labelstyle=above,style={->}](c)(x)
\end{tikzpicture}
  \end{minipage}
\caption{Out-neighborhoods of the central vertex in quotients of $S_3$.}
  \label{quotient3stars}
\end{figure}

Because all row sums in the adjacency matrix of a $k$-orbigraph $\og$ are $k$, the out-neighborhood of a vertex $v$ in $\og$ is identical to the out-going neighborhood of the central vertex in some quotient of a $k$-star.  In this way, a $k$-star quotient provides the local model of the neighborhood of a point in an orbigraph.  Our interest in the local structure of an orbigraph at a vertex is in the number of outgoing edges and the weights of those edges.  The terminal point of an outgoing edge is not important.  Because of this the out-neighborhood of a vertex with a loop is taken with the loop `undone.'  For example, vertex $v_1$ in Figure \ref{smallexa} is locally modeled on the middle graph in Figure \ref{quotient3stars}.

To complete our analogy between the local structure of orbifolds and the local structure of orbigraphs we observe that requirement (iii) in Definition \ref{def:OrbigraphDefn} corresponds to the fact that if local neighborhoods $U, V$ in an orbifold satisfy $U\cap V \ne \emptyset$ then we also have $V \cap U \ne \emptyset$.

\subsection{Singular points in an orbigraph}

The key feature of the study of orbifolds that distinguishes it from manifold theory is the presence of orbifold singular points.  We define a singular vertex in an orbigraph in the following way.
\begin{definition}\label{def:SingularPoint}  A vertex $v$ of an orbigraph is \emph{singular} if any outgoing edge from $v$ has weight greater than one.  A vertex that is not singular is called \emph{regular}.
\end{definition}
We see that regular graphs contain no singular vertices, as required by our analogy between regular graphs and manifolds.

\begin{example} Both vertices in the orbigraph in Figure \ref{smallexa} are singular.  Vertices $v_1, v_4$ and $v_6$ in the orbigraph in Figure \ref{bigexa} are singular, and the rest are regular. 
\end{example}  

In contrast to the orbifold setting, singular points in an orbigraph are not marked with an isotropy group.  However we can quantify the extent to which a vertex $v$ is singular by noting the number of outgoing edges from $v$ that have weight greater than one.  We can also consider the list of weights of outgoing edges from $v$.  As mentioned in the introduction, graphs-of-groups offer an alternative graph theoretic interpretation of  orbifolds.  A graph-of-groups, in contrast to an orbigraph, has vertices that are marked with a group in a way that is analogous to an orbifold isotropy group.

\subsection{Good and bad orbigraphs}

In Example \ref{exa:football} we saw that the football orbifold is the quotient of a sphere under the smooth action of a finite group.  In Example \ref{exa:teardrop} it was asserted that the teardrop orbifold cannot be obtained as a quotient in this manner. Orbifolds that can be written as the quotient of a manifold under a smooth, discrete group action are called \emph{good}.  Otherwise they are called \emph{bad}.  Following these ideas we define \emph{good} and \emph{bad} orbigraphs as follows.

\begin{definition}  A $k$-orbigraph $\Omega$ is said to be \emph{good} if it can be obtained as the quotient of a finite $k$-regular graph $\Gamma$ via an equitable partition on $\Gamma$.  If an orbigraph is not good it is called \emph{bad}.
\end{definition}

\begin{example}\label{exa:goodbad}  The orbigraph in Figure 3 is good because it is the quotient of the complete graph $K_4$, as presented in Figure \ref{k4}, by the group $\Z_3$ generated by a $2\pi/3$ radian rotation about the center vertex.  The orbigraph in Figure 4 is bad. This follows from Theorem \ref{thm:balancedcycle} below and the observation that the product of edge weights along cycle $(v_1,v_2,v_3,v_4,v_5,v_6,v_7,v_1)$ is two, while the product of edge weights along the reverse cycle $(v_1,v_7,v_6,v_5,v_4,v_3,v_2,v_1)$ is four.
\end{example}

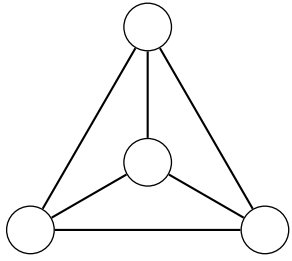
\begin{figure}
\centering
\begin{tikzpicture}
[scale=.6]
\GraphInit[vstyle=Normal]
\SetGraphUnit{3}
\begin{scope}[rotate=90]
\Vertices[Math,NoLabel]{circle}{v_1,v_2,v_3}
\Edge(v_1)(v_2)
\Edge(v_2)(v_3)
\Edge(v_3)(v_1)
\end{scope}
\Vertex[Math,NoLabel]{v_4}
\Edge(v_3)(v_4)
\Edge(v_1)(v_4)
\Edge(v_2)(v_4)
\end{tikzpicture}
\caption{Graph diagram of $K_4$.}
  \label{k4}
\end{figure}

The analogy with the covering theory of topological spaces is further strengthened by the following two lemmas. 

\begin{lemma}\label{lemma:EqPartitionQuotient}  If $\og$ is a $k$-orbigraph and ${\cal P}$ is an equitable partition on the vertices of $\og$, then $\og / {\cal P}$ is a $k$-orbigraph.
    \end{lemma}
    \begin{proof} Let $A$ denote the adjacency matrix of $\og / {\cal P}$, partition ${\cal P}=\{V_1, V_2, \dots, V_m\}$, and $w_\og(\cdot,\cdot)$ denote the weight function on directed edges in $\og$. Because $\og$ is an orbigraph, we know $w_\og(u, v)$ is a nonnegative integer for all vertices $u, v$ in $\og$. Hence $A_{ij} = \sum_{v \in V_j} w_\og(u,v)$, for any $u \in V_i$, is a nonnegative integer.  Fixing $i\in\{1, 2, \dots, m\}$,  and taking $u$ some element of $V_i$, consider the $i$th row sum of $A$:  
      \begin{align*}
        \sum_j A_{ij} &= \sum_j \sum_{v \in V_j} w_\og(u,v) \\
        &= \sum_{v \in \og} w_\og(u, v) = k.
      \end{align*}

      Finally suppose $A_{ij} > 0$. Then there must a $j\in\{1, 2, \dots, m\}$ for which any $u \in V_i$ has $w_\og(u,v)>0$ for some $v \in V_j$. Because $\og$ is an orbigraph, we must also have $w_\og(v, u) > 0$. Thus $A_{ji} > 0$.
    \end{proof}

\begin{definition}\label{def:cover}
We say that an orbigraph $\og_1$ covers an orbigraph $\og_2$ if there is an equitable partition ${\cal P}$ of the vertices of $\og_1$ such that $\og_1/{\cal P}= \og_2$.
\end{definition}

    \begin{lemma}\label{lemma:CoveringTransitivity}
      The covering relation is transitive.
    \end{lemma}
    \begin{proof} Suppose $\og_1$ is an orbigraph with equitable partition $\cP_1$ such that $\og_1 / \cP_1 = \og_2$, and $\og_2$ has an equitable partition $\cP_2$ such that $\og_2 / \cP_2 = \og_3$. We need to show there is an equitable partition $\cP_3$ of $\og_1$ such that $\og_1 / \cP_3= \og_3$. For $i=1,2$ let $A_i$ denote the adjacency matrix of orbigraph $\og_i$, and $P_i$ denote the characteristic matrix corresponding to partition $\cP_i$. By a straightforward modification of Godsil \cite[Lemma 2.1, p.~77]{MR1220704} to the setting of weighted, directed graphs we have that $A_1P_1 = P_1A_2$ and $A_2P_2 = P_2A_3$. Thus  $A_1P_1P_2 = P_1A_2P_2=P_1P_2A_3$.   We conclude $P_1P_2$ defines an equitable partition on $\og_1$ with quotient orbigraph $\og_3$.
\end{proof}

As a consequence of the previous two lemmas we obtain the following.

\begin{corollary}\label{cor:allgood} The quotient of any good orbigraph must also be good.
\end{corollary}

\section{Orbigraphs and Markov chains}\label{markov}

The fact that the row sum of the adjacency matrix of an orbigraph is constant provides an immediate connection between orbigraphs and Markov chains.  Following Kelly \cite{MR554920}, we review ideas from the theory of Markov chains and introduce notation that will be used hereafter.  Matrix $A$ will denote the adjacency matrix of a $k$-orbigraph $\og$ with $n$ vertices.   Define $P=\frac{1}{k}A$.  Matrix $P$ is the transition matrix of a stationary Markov chain as all entries of $P$ lie in the interval $[0,1]$ and all rows of $P$ sum to 1.  Because the adjacency matrix of a $k$-orbigraph has right eigenvalue $k$ (to see this consider the eigenvector with all entries equal to one), $P$ has right eigenvalue one and stationary distribution vector $\pi=(\pi_1, \pi_2, \dots, \pi_n)$ with $\sum_{k=1}^n \pi_k =1$ for which $\pi P = \pi$.  By Remark~\ref{connected} we know $\Omega$ is strongly connected so $\pi$ is the unique stationary distribution of $P$.  

Our first result connecting orbigraphs to Markov chains is a bound on the minimal entry of $\pi$ in terms of the degree and number of vertices of an orbigraph.

\begin{lemma}\label{minentry} Let $\pi_m$ be a minimal entry in stationary distribution $\pi$.  Then $$\pi_m \ge \frac{1}{nk^{n-1}} \ .$$
\end{lemma}

\begin{proof} Let $\pi_M$ denote a maximal entry in $\pi$ and let $c$ be the minimal nonzero value that appears as an entry in matrix $P$.  Because $\og$ is strongly connected there is a path of length $\ell < n$ from the $M$th vertex to the $m$th vertex of $\og$.  This implies that $(P^\ell)_{Mm}$ is nonzero.  Using this and the fact that $\pi P = \pi$ we have,
\begin{align*}
\pi_m &= \sum_{k=1}^n (P^\ell)_{km} \pi_k \\
& \ge (P^\ell)_{Mm} \pi_M \\
& \ge c^\ell \pi_M \\
& \ge c^{n-1} \pi_M.
\end{align*}
Because $P$ is the transition matrix associated to an orbigraph we have $c\ge \frac{1}{k}$.  Also, we know that $\pi_M \ge \frac{1}{n}$ because the sum of the entries of $\pi$ is one.  Thus $\pi_m \ge c^{n-1} \pi_M \ge \frac{1}{nk^{n-1}}$ as required.
\end{proof}

Here we relate the stationary distribution of a good orbigraph to that of its finite regular cover.

\begin{lemma}\label{lump}  Let $\Gamma$ be a $k$-regular graph with $N$ vertices, ${\cal P}= \{V_1, V_2, \dots, V_n\}$ be an equitable partition of the vertices of $\Gamma$, and $P$ be the transition matrix of the orbigraph $\Gamma/\cal{P}$.  Let $|V_i|$ denote the number of vertices in partition element $V_i$. The stationary distribution of $P$ is the $n$-tuple $\pi$ where $\pi_i=\frac{1}{N}|V_i|$. 
\end{lemma}

\begin{proof}  Let $Q$ denote the transition matrix obtained by scaling the adjacency matrix of $\Gamma$ by $\tfrac{1}{k}$.  The result  follows from the observation that the stationary distribution of $Q$ is the $N$-tuple $(\frac{1}{N},\frac{1}{N}, \dots, \frac{1}{N})$ and Godsil \cite[Lemma 2.2, p.78]{MR1220704}.
\end{proof}

\section{Characterizing good and bad orbigraphs}\label{bcc}

We use the Markov chain methods and notation from Section \ref{markov} to provide a quick way to distinguish good orbigraphs from bad orbigraphs.

\begin{definition}\label{def:balancedcycle}  An orbigraph $\og$ satisfies the \emph{balanced cycle condition} if the product of the edge weights along each directed cycle $v_1, v_2,\dots, v_l,v_1$ in $\og$ equals the product of the edge weights along the reverse directed cycle $v_1, v_{l}, v_{l-1}, \dots , v_1$. 
\end{definition}

\begin{theorem}\label{thm:balancedcycle}  An orbigraph is good if and only if it satisfies the balanced cycle condition.
\end{theorem}

A stationary Markov chain is said to satisfy the \emph{detailed balance equations} if
\[\pi_i P_{ij} = \pi_j P_{ji} \ \ \text{for all} \ i,j=1,2,\dots, n.\]

The Markov chain analog of the balanced cycle condition from Definition \ref{def:balancedcycle} is called the \emph{Kolmogorov criterion}.  In particular, an orbigraph satisfies the balanced cycle condition if and only if the corresponding Markov chain satisfies the Kolmogorov criterion.  We can now state a needed lemma.

\begin{lemma}\label{lem:dbeKolm}  A stationary Markov chain satisfies the detailed balance equations if and only if it satisfies Kolmogorov's criteria.
\end{lemma}

\begin{proof}  This follows from combining Theorems 1.2 and 1.7 in Kelly \cite{MR554920}.
\end{proof}

\begin{proof}[Proof of Theorem \ref{thm:balancedcycle}]
Suppose $\Omega$ is a good orbigraph.  This implies $\Omega = \Gamma/{\cal P}$ where $\Gamma$ is a $k$-regular graph and ${\cal P}= \{V_1, V_2, \dots, V_n\}$ is an equitable partition on $\Gamma$.  Scaling the adjacency matrix of $\Gamma$ by $1/k$ yields the symmetric transition matrix $Q$ of a Markov chain.   We relate the stationary distribution of $Q$ to the stationary distribution of $P$, the transition matrix of $\Omega$,  by Lemma~\ref{lump}.  In particular $\pi_i=\frac{1}{N}|V_i|$, where $\pi$ denotes the stationary distribution of $P$ and $N$ is the number of vertices in $\Gamma$.  

The following computation confirms that $P$ satisfies the detailed balance equations.  The argument closely follows that of Tian and Kannan \cite[Theorem 2.16]{MR2220077} which is given in the setting of lumpable Markov chains.  It makes essential use of the fact that $\cal P$ is an equitable partition and that $Q$ is a symmetric matrix.
\begin{align*}
\pi_j P_{ji} &= \tfrac{1}{N}|V_j| P_{ji} \\
&= \tfrac{1}{N} |V_j|  \sum_{k\in V_i} Q_{jk} \\
&= \tfrac{1}{N} \sum_{l\in V_j} \sum_{k\in V_i} Q_{lk} \\
&= \tfrac{1}{N} \sum_{k\in V_i} \sum_{l\in V_j} Q_{kl} \\
&= \tfrac{1}{N} |V_i| \sum_{l\in V_j} Q_{kl} \\
&= \pi_i P_{ij}
\end{align*}
The fact that $\Omega$ satisfies the balanced cycle condition now follows from Lemma~\ref{lem:dbeKolm}.

Now suppose $\Omega$ is an orbigraph that satisfies the balanced cycle condition. By Lemma \ref{lem:dbeKolm}, $P$ and $\pi$ satisfy the detailed balance equations $\pi_i P_{ij} = \pi_j P_{ji} $.  Multiplying by $k$ on both sides gives $\pi_i A_{ij} = \pi_j A_{ji}$.  Because $A$ has all non-negative integer entries, $\pi$ will have all non-negative rational entries.  Thus there is an integer $m$ for which $m\pi = (d_1, d_2, \dots, d_n)$ is a vector of non-negative integers.  This allows us to write 
\begin{align}\label{normdbe}
d_i A_{ij} = d_j A_{ji},
\end{align}
an equality of products of non-negative integers.

We now build a finite $k$-regular cover $\Gamma$ of $\Omega$.  Let $X$ be the set of non-zero, non-diagonal entries of $A$.  Let $Y = \{A_{11}+1, A_{22}+1, \dots, A_{nn}+1\}$.  Let $c$ be the least common multiple of the integers in $X \cup Y$.  For each $i=1,2,\dots, n$ we take $V_i$ to be a set of $cd_i$ vertices.   The disjoint union $V_1 \sqcup V_2 \sqcup \dots \sqcup V_n$ forms the vertex set of $\Gamma$ and gives the needed vertex partition ${\cal P}$ of $\Gamma$. 

It remains to specify adjacency in $\Gamma$ in such a way that $\Gamma/{\cal P}=\og$.  Suppose $i\ne j$.  For the quotient $\Gamma/{\cal P}=\og$ to be valid each vertex in $V_i$ must be adjacent to $A_{ij}$ vertices in $V_j$, and each vertex in $V_j$ must be adjacent to $A_{ji}$ vertices in $V_i$.  Thus the number of edges  with one vertex in $V_i$ and one vertex in $V_j$, which we will denote by $e_{\{i,j\}}$, is simultaneously $A_{ij}|V_i|$ and $A_{ji}|V_j|$. The adapted detailed balance equations from Line \ref{normdbe} show that this requirement follows from our choice for the sizes of $V_i$ and $V_j$ as,
\begin{align*}
A_{ij}|V_i|=A_{ij}cd_i = A_{ji}cd_j = A_{ji}|V_j|.
\end{align*}
Because $A_{ij}$ divides $|V_j|$ and $A_{ji}$ divides $|V_i|$, we can distribute the $e_{\{i,j\}}$ edges connecting $V_i$ and $V_j$ with exactly $A_{ij}$ edges adjacent to each vertex in $V_i$ and exactly $A_{ji}$ edges adjacent to each vertex in $V_j$.  Because $(A_{ii}+1)$ divides $|V_i|$ we can require that all elements of $V_i$ are adjacent to exactly $A_{ii}$ other elements of $V_i$.  This completes the adjacency relations for $\Gamma$.

By construction we observe $\Gamma/{\cal P}=\og$.  The degree of a vertex $v$ in $\Gamma$ is $\sum_{j=1} A_{ij} =k$, thus $\Gamma$ is $k$-regular.  Should $\Gamma$ fail to be connected, any connected component $\Gamma'$ of $\Gamma$ will satisfy $\Gamma'/\cP = \og$.
\end{proof}

\begin{remark} Corollary \ref{cor:allgood} and Theorem \ref{thm:balancedcycle} imply that if an orbigraph $\og$ satisfies the balanced cycle condition then so does any orbigraph quotient of $\og$.  This stands in contrast to Tian and Kannan \cite[Example 2.17]{MR2220077}.
\end{remark}

\section{Bounding the Cheeger constant of an orbigraph}\label{cheeger}

Chung \cite{MR2135772} defined a Cheeger constant for directed graphs and obtained lower bounds on the Cheeger constant for both regular and Eulerian directed graphs.  Using $R$ to denote a $k$-regular directed graph on $n$ vertices and $E$ an Eulerian directed graph with $m$ edges, Chung showed
\begin{align} h(R) \ge \frac{2}{kn} \ \ \  \text{and} \ \ \ h(E) \ge \frac{2}{m}.\end{align}\label{chungineq}
Here we apply Chung's methods to obtain a lower bound on the Cheeger constant of an orbigraph.  We use notation from Section \ref{markov}. 

Define a function $F$ from $\Omega$ to the non-negative real numbers by
$$F(i,j) = \pi_i P_{ij}$$
where $i$ and $j$ are vertices in $\Omega$. This function is an example of a \emph{circulation} on $\Omega$ (see Chung \cite[Lemma 3.1]{MR2135772}).  Letting $S$ range over all non-empty proper subsets of the vertex set of $\Omega$, the Cheeger constant $h(\Omega)$ of $\Omega$ is defined as
\[h(\Omega)= \inf_S \frac{\sum_{i\in S, j \notin S} F(i,j)}{\min\left\{\sum_{j\in S} F(j), \sum_{j\in \bar{S}} F(j)\right\}} \]
where $F(j)=\sum_{i, i \rightarrow j} F(i,j)$ and $\bar{S}$ is the set of vertices of $\Omega$ that are not in $S$.

We have the following lower bound on the Cheeger constant of $\Omega$.

\begin{proposition} Let $\Omega$ be a $k$-orbigraph with $n$ vertices. Then 
\[h(\Omega)\ge \frac{2}{n^2k^n} \ .\]
\end{proposition}

\begin{proof}  We begin by bounding the numerator in the expression defining the Cheeger constant.  Let $\pi_m$ denote a minimal entry in $\pi$.
\begin{align*}
\sum_{i\in S, j \notin S} F(i,j) &= \sum_{i\in S, j \notin S} \pi_i P_{ij}\\
&\ge \sum_{i\in S, j \notin S} \pi_m P_{ij} \\
&\ge \frac{1}{nk^{n}}. 
\end{align*}
The last line follows from Lemma~\ref{minentry} and the observation that the smallest  possible nonzero value for an entry in $P$ is $\tfrac{1}{k}$.

To bound the denominator first observe that $\sum_{j\in S} F(j)$ is no greater than the sum of the columns in $P$ associated to the vertices in $S$. Similarly with $\sum_{j\in \bar{S}} F(j)$.  Since the total sum of the entries in $P$ is $n$ we have
\[\sum_{j\in S} F(j)+\sum_{j\in \bar{S}} F(j) \le n.\]
Thus $\min\left\{\sum_{v\in S} F(v), \sum_{v\in \overline{S}} F(v)\right\} \le \frac{n}{2}$.

We see that for any choice of $S$ the quotient in the definition of the Cheeger constant must be greater than or equal to $\frac{2}{n^2k^n}$, completing the proof.
\end{proof}

\begin{remark} Chung uses the inequalities in (2) to obtain convergence bounds for a type of random walk on regular and Eulerian directed graphs. The presence of $n$ in the exponent in the denominator of the orbigraph bound make it too weak to obtain a similar orbigraph result.  It would be interesting to see if a better bound on the Cheeger constant of an orbigraph, should one exist, would allow a convergence result similar to the regular and Eulerian cases.
\end{remark}

\section{Spectral results for orbigraphs}\label{spectral}

Because different matrices can be associated to a given graph, a variety of graph spectra are examined in spectral graph theory.  Here the  \emph{spectrum} of an orbigraph $\og$ is defined to be the list of eigenvalues of the adjacency matrix of $\og$ with each eigenvalue repeated according to its multiplicity.  We will write the spectrum of an orbigraph with $n$ vertices as a multiset  $\{\lambda_1, \lambda_2, \dots, \lambda_n\}$.  The study of the spectral properties of directed graphs is relatively new and has yielded interesting applications as well as directed graph analogs of familiar graph theoretical results including Cheeger's Inequality (see  Chung \cite{MR2135772}, and Langville and Meyer \cite{MR2262054}, for example).  We focus on developing results that relate the spectrum of an orbigraph to its orbigraph structure.

\begin{remark}\label{rem:easyspec} Just as with $k$-regular graphs, the spectral radius of a $k$-orbigraph is $k$. In addition the number of eigenvalues in the spectrum of an orbigraph (counting multiplicity) is equal to the number of vertices in the orbigraph. 
\end{remark}

\begin{lemma}\label{lem:speccontain}  Suppose orbigraph $\Omega_1$ covers orbigraph $\og_2$.  Then the spectrum of $\Omega_2$ is contained in the spectrum of $\og_1$ as multisets.
\end{lemma}

\begin{proof}  This follows from the argument in Lemma 2.2 of Chapter 5 in Godsil \cite{MR1220704}, adjusted to allow the graph carrying the equitable partition to be a weighted, directed graph.
\end{proof}

\begin{corollary}  Any orbigraph with complex eigenvalues must be bad.
\end{corollary}

\begin{proof} This follows from Lemma \ref{lem:speccontain} and the fact that regular graphs have real eigenvalues.
\end{proof}

\begin{theorem}\label{specgoodbad} The spectrum of an orbigraph does not distinguish good orbigraphs from bad orbigraphs.
\end{theorem}
\begin{proof}  The orbigraph on the left in Figure \ref{goodbadcosp} and the orbigraph in the center of the Figure both have spectrum $\{-2,0,1,3\}$.  However the orbigraph on the left is bad and the orbigraph in the center is good.  To see that the left hand orbigraph is bad apply Theorem \ref{thm:balancedcycle} and the fact that the product of the edge weights along cycle $(v_1, v_2, v_3, v_4)$ is not equal to the product of the edge weights of this cycle reversed.  The center orbigraph is good because it is covered by the $3$-regular graph on the right side of Figure  \ref{goodbadcosp} using the indicated equitable partition.  
\end{proof}

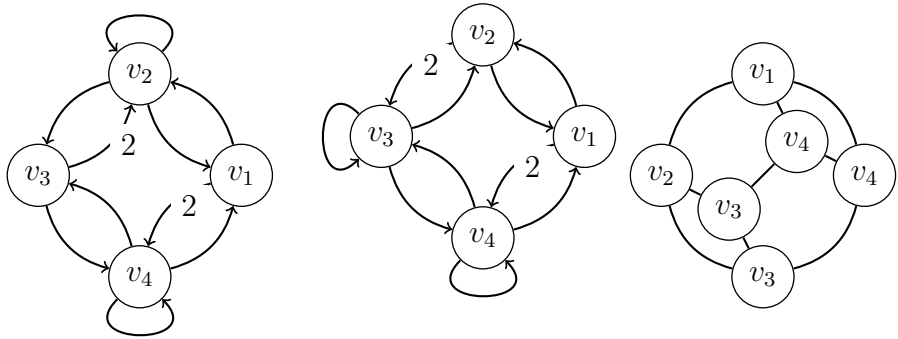
\begin{figure}

\centering
\begin{minipage}{.3\textwidth}
 \centering
\begin{tikzpicture}
[scale=.45]
\GraphInit[vstyle=Normal]
\SetGraphUnit{3}
\Vertices[Math]{circle}{v_1,v_2,v_3,v_4}
\Edge[style={->,bend right}](v_1)(v_2)
\Edge[style={->,bend right}](v_2)(v_3)
\Loop[dist=2cm,dir=NO,style={->}](v_2)
\Loop[dist=2cm,dir=SO,style={->}](v_4)
\Edge[style={->,bend right}](v_2)(v_1)
\Edge[style={->,bend right},](v_3)(v_4)
\Edge[style={->,bend right},label=2,labelstyle={right,fill=white}](v_3)(v_2)
\Edge[style={->,bend right}](v_4)(v_1)
\Edge[style={->,bend right}](v_4)(v_3)
\Edge[style={->,bend right},label=2,labelstyle={right,fill=white}](v_1)(v_4)
\end{tikzpicture}
\
\end{minipage}
\begin{minipage}{.3\textwidth}
\begin{tikzpicture}
[scale=.45]
\GraphInit[vstyle=Normal]
\SetGraphUnit{3}
\Vertices[Math]{circle}{v_1,v_2,v_3,v_4}
\Edge[style={->,bend right}](v_1)(v_2)
\Edge[style={->,bend right},label=2,labelstyle={right,fill=white}](v_2)(v_3)
\Loop[dist=2cm,dir=WE,style={->}](v_3)
\Loop[dist=2cm,dir=SO,style={->}](v_4)
\Edge[style={->,bend right}](v_2)(v_1)
\Edge[style={->,bend right},](v_3)(v_4)
\Edge[style={->,bend right}](v_3)(v_2)
\Edge[style={->,bend right}](v_4)(v_1)
\Edge[style={->,bend right}](v_4)(v_3)
\Edge[style={->,bend right},label=2,labelstyle={right,fill=white}](v_1)(v_4)
\end{tikzpicture}
\
\end{minipage}
\begin{minipage}{.3\textwidth}
\begin{tikzpicture}
[scale=.45]
\GraphInit[vstyle=Normal]
\SetGraphUnit{3}
\Vertex[Math,L=v_3,x=-1,y=-1]{v_5}
\Vertex[Math,L=v_4,x=1,y=1]{v_6}
\Vertex[Math,x=3,y=0,L=v_4]{v_1}
\Vertex[Math,L=v_1,x=0,y=3]{v_2}
\Vertex[Math,x=0,y=-3,L=v_3]{v_4}
\Vertex[Math,x=-3,y=0,L=v_2]{v_3}
\Edge[style={-,bend right}](v_1)(v_2)
\Edge[style={-,bend right}](v_2)(v_3)
\Edge[style={-,bend right}](v_3)(v_4)
\Edge[style={-,bend right}](v_4)(v_1)
\Edge(v_5)(v_6)
\Edge(v_5)(v_3)
\Edge(v_5)(v_4)
\Edge(v_6)(v_1)
\Edge(v_6)(v_2)
\end{tikzpicture}

\end{minipage}

\caption{The left and center orbigraphs are cospectral.  The left orbigraph is bad.  The center orbigraph is good as it is covered by the right-most graph using the indicated partition.}
  \label{goodbadcosp}
\end{figure}

In the following lemma a directed edge from vertex $v_1$ to vertex $v_2$ of weight $w$ is considered to contribute $w$ many different ways to move from $v_1$ to $v_2$.  The length spectrum of a graph is the finite list of non-negative integers where the $m$th number in the list counts the number of closed walks of length $m$ present in the graph.

\begin{lemma}\label{thm:LengthSpectrum}
      The eigenvalue spectrum of an orbigraph determines and is determined by the length spectrum of the orbigraph.
    \end{lemma}
    
    \begin{proof}
      Let $\og$ be a $k$-orbigraph, $A$ its adjacency matrix, and $w_m$ the number of closed walks in $\og$ of length $m$. We know that
        \begin{equation}\label{eq:TraceLengthSpectrum}
          w_m = \tr(A^m)
        \end{equation}
      because the diagonal of $A^m$ counts the number of closed walks of length $m$. However
      \begin{equation*}
        \tr(A^m) = \sum_{i=1}^n \lambda_i^m.
      \end{equation*}
      Thus the eigenvalue spectrum of $\og$ uniquely determines the length spectrum of $\Omega$, and conversely by Newton's identities \cite{MR1186460} the length spectrum of $\Omega$ uniquely determines the eigenvalue spectrum of $\og$.
    \end{proof}

We now prove that the number of singular points in an orbigraph is bounded above and below by spectrally determined quantities.

\begin{theorem}\label{thm:SingularBounds}
      Let $\Omega$ be an $k$-orbigraph with $n$ vertices. If $s$ is the number of singular points in $\Omega$, then we have
      $$
        \frac{\sum_{i=1}^n \lambda_i^2 - n k}{k^2 - k} \le s \le \sum_{i=1}^n \lambda_i^2 - n k
      $$
      where $\lambda_i$ are the eigenvalues of the adjacency matrix $A$ of $\Omega$.
    \end{theorem}

\begin{proof} First note that $\sum_{i=1}^n \lambda_i^2 = \tr(A^2)$ and by Lemma \ref{thm:LengthSpectrum} this quantity counts the number of closed walks of length $2$ in $\Omega$. A given vertex $v$ in $\Omega$ has outgoing edges with weights summing to $k$, each of which is matched by at least one incoming edge.  This implies the number of closed walks of length $2$ starting at $v$ is at least $k$.  Observing that there are $n$ vertices in $\Omega$ we obtain $\tr(A^2) \ge n k$.  Now suppose $v_1$ is a singular vertex in $\Omega$.  This vertex has at least one outgoing edge $(v_1,v_2)$ of weight greater than 1.  Edge $(v_1,v_2)$ contributes at least one closed walk of length two, beginning and ending at $v_2$, that has not yet been counted.  We conclude that  $\tr(A^2) \ge n k + s$ thus  $s \le \sum_{i=1}^n \lambda_i^2 - n k$. 

For the lower bound, note that each singular vertex $v_i$ contributes $A_{ji} (A_{ij} - 1)$ extra (i.e. beyond the initial $k$ length-two paths) length-two paths based at $v_j$.  Thus the total number of extra paths contributed by vertex $v_i$ is $\sum_{v_i \sim v_ j} A_{ji} (A_{ij} - 1)$.  We bound this quantity in terms of $k$,
\begin{align*}
\sum_{v_i \sim v_ j} A_{ji} (A_{ij} - 1)  &\le  \sum_{v_i \sim v_ j} k (A_{ij} - 1) \\
&= k \sum_{v_i \sim v_j} A_{ij} - \sum_{v_i \sim v_j}  k \\
&\le k^2 - k.
\end{align*}
Hence each singular vertex contributes at most $k^2 - k$ extra walks of length two, so $s(k^2 - k) \ge \sum_{i=1}^n \lambda_i^2 - n k$.  Isolating $s$ in this inequality completes the proof.

\end{proof}

\begin{remark}    The orbigraph with adjacency matrix $kI_n$, where $I_n$ denotes the $n\times n$ identity matrix, achieves the lower bound in Theorem \ref{thm:SingularBounds} for all choices of $k$ and $n$.  Thus this lower bound  is sharp in $k$ and $n$.
\end{remark}

\begin{corollary}\label{cor:charkreg}  Suppose $\og$ is a $k$-orbigraph with $n$ vertices.  Then $\og$ is isomorphic to a $k$-regular graph if and only if 
\[\sum_i \lambda_i^2 -nk = 0 \ \  \text{and} \ \ \sum_i \lambda_i =0.\]
\end{corollary}

\begin{proof}
 A simple $k$-regular graph $\og$ has no self loops, thus Lemma \ref{thm:LengthSpectrum} implies $\sum_i \lambda_i =0$.  Viewing each edge $\{v_i,v_j\}$ in $\og$ as two directed edges, $(v_i,v_j)$ and $(v_j,v_i)$, we see each vertex in $\og$ has exactly $k$ closed walks of length $2$.  Therefore $\sum_{i} \lambda_i^2 = n k$.

Conversely, assume that $\og$ is an orbigraph such that $\sum_{i} \lambda_i^2 = n k$ and $\sum_{i} \lambda_i = 0$. Then by Theorem \ref{thm:SingularBounds}, we have $s \le 0$. As $s \ge 0$ we see $s = 0$. Thus the outgoing edges of each vertex in $\og$ all have weight one. The second condition implies $\og$ has no loops.  By combining pairs of directed edges  $(v_i,v_j)$ and $(v_j,v_i)$ into a single undirected edge $\{v_i,v_j\}$ we obtain a simple $k$-regular graph.
\end{proof}

In the smooth setting it is not known if a manifold can have the same Laplace spectrum as a non-manifold orbifold.  We can resolve this question in the setting of orbigraphs.
\begin{corollary}  A regular graph and an orbigraph with one or more singular points cannot be cospectral.  
\end{corollary}

\begin{proof} Suppose regular graph $\Gamma$ and orbigraph $\og$ are cospectral and that $\og$ contains $s\ge 1$ singular points.  By Remark \ref{rem:easyspec} the largest eigenvalue in the shared spectrum of  $\Gamma$ and $\og$ is the degree of regularity of each graph.  Denote this largest eigenvalue by $k$. In addition the shared spectrum implies that each graph has the same number of vertices $n$.  By the forward direction of Corollary \ref{cor:charkreg} the fact that $\Gamma$ is $k$-regular implies $\sum_i \lambda_i^2 -nk = 0$ and $\sum_i \lambda_i =0$.  However the backwards direction of Corollary \ref{cor:charkreg} implies $s=0$, a contradiction.
\end{proof}

\bibliographystyle{plain}
\bibliography{orbigraphsbib}

\end{document}